\documentclass[11pt,reqno]{amsart}
\calclayout

\makeatletter
\newcommand{\monthyear}[1]{%
  \def\@monthyear{\uppercase{#1}}}
\newcommand{\volnumber}[1]{%
  \def\@volnumber{\uppercase{#1}}}
\AtBeginDocument{%
\def\ps@plain{\ps@empty
  \def\@oddfoot{\@monthyear \hfil \thepage}%
  \def\@evenfoot{\thepage \hfil \@volnumber}}
\def\ps@firstpage{\ps@plain}
\def\ps@headings{\ps@empty
  \def\@evenhead{%
    \setTrue{runhead}%
    \def\thanks{\protect\thanks@warning}%
\uppercase{The Fibonacci Quarterly}\hfil}%
  \def\@oddhead{%
    \setTrue{runhead}%
    \def\thanks{\protect\thanks@warning}%
    \hfill\uppercase{Deterministic Zeckendorf Games}}%
  \let\@mkboth\markboth
  \def\@evenfoot{%
    \thepage \hfil \@volnumber}%
  \def\@oddfoot{%
    \@monthyear \hfil \thepage}%
  }%
\footskip=25pt
\pagestyle{headings}%
}
\makeatother

\theoremstyle{plain}
\numberwithin{equation}{section}
\newtheorem{thm}{Theorem}[section]

\newtheorem{lemma}[thm]{Lemma}

\usepackage{amsmath,amsthm,amssymb,comment}
\usepackage{braket}
\usepackage{mathtools}
\usepackage{caption}
\usepackage{times}
\usepackage[T1]{fontenc}
\usepackage{mathrsfs}
\usepackage{latexsym}
\usepackage[dvips]{graphics}
\usepackage{epsfig}
\usepackage{amsmath,amsfonts,amsthm,amssymb,amscd}
\input amssym.def
\input amssym.tex
\usepackage{hyperref}
\usepackage{color}
\newcommand{\bburl}[1]{\textcolor{blue}{\url{#1}}}

\usepackage{tikz}
\usepackage{tkz-tab}
\usepackage{tkz-graph}
\usetikzlibrary{shapes.geometric,positioning}
\newcommand{\burl}[1]{\textcolor{blue}{\url{#1}}}

\numberwithin{equation}{section}

\newcommand\be{\begin{equation}}
\newcommand\ee{\end{equation}}
\newcommand\bee{\begin{equation*}}
\newcommand\eee{\end{equation*}}
\newcommand\bea{\begin{eqnarray}}
\newcommand\eea{\end{eqnarray}}
\newcommand\beae{\begin{eqnarray*}}
\newcommand\eeae{\end{eqnarray*}}
\newcommand\bi{\begin{itemize}}
\newcommand\ei{\end{itemize}}
\newcommand\ben{\begin{enumerate}}
\newcommand\een{\end{enumerate}}
\newcommand\bc{\begin{center}}
\newcommand\ec{\end{center}}
\newcommand\ba{\begin{array}}
\newcommand\ea{\end{array}}
\newcommand\frakfamily{\usefont{U}{yfrak}{m}{n}}
\DeclareTextFontCommand{\textfrak}{\frakfamily}

\newtheorem{rek}[thm]{Remark}
\newtheorem{defi}[thm]{Definition}
\newtheorem{conj}[thm]{Conjecture}

\newcommand{\hr}[1]{\href{#1}{\url{#1}}}

\begin{document}
%% replace the values in the next three lines by the correct information
\monthyear{Month Year}
\volnumber{Volume, Number}
\setcounter{page}{1}

\title{Deterministic Zeckendorf Games}

\author{Ruoci Li}
\address{Elite Scholars Program, New York, NY}
\email{\textcolor{blue}{\href{mailto:ruocili2002@yahoo.com}{ruocili2002@yahoo.com}}}

\author{Xiaonan Li}
\address{L\'eman Manhattan Preparatory School, NY}
\email{\textcolor{blue}{\href{mailto:lxn020905@gmail.com}{lxn020905@gmail.com}}}

\author{Steven J. Miller}
\address{Department of Mathematics and Statistics, Williams College, Williamstown, MA 01267}
\email{\textcolor{blue}{\href{mailto:sjm1@williams.edu}{sjm1@williams.edu}}}

\author{Clay Mizgerd}
\address{Department of Mathematics and Statistics, Williams College, Williamstown, MA 01267}
\email{\textcolor{blue}{\href{mailto:cmm12@@williams.edu}{cmm12@@williams.edu}}}

\author{Chenyang Sun}
\address{Department of Mathematics and Statistics, Williams College, Williamstown, MA 01267}
\email{\textcolor{blue}{\href{mailto:cs19@williams.edu}{cs19@williams.edu}}}

\author{Dong Xia}
\address{Ross School, East Hampton, NY}
\email{\textcolor{blue}{\href{mailto:dongxia08@yahoo.com}{dongxia08@yahoo.com}}}

\author{Zhyi Zhou}
\address{Elite Preparatory Academy, NJ}
\email{\textcolor{blue}{\href{mailto:zyzhou00@yahoo.com}{zyzhou00@yahoo.com}}}

\maketitle

\begin{abstract} Zeckendorf \cite{Ze} proved that every positive integer can be written uniquely as the sum of non-adjacent Fibonacci numbers. We further explore a two-player Zeckendorf game introduced in \cite{BEFMfibq, BEMFcant}: Given a fixed integer $n$ and an initial decomposition of $n = nF_1$, players alternate using moves related to the recurrence relation $F_{n+1}  = F_n + F_{n_1}$, and the last player to move wins. We improve the upper bound on the number of moves possible and show that it is of the same order in $n$ as the lower bound; this is an improvement by a logarithm over previous work. The new upper bound is $3n - 3Z(n) - IZ(n) + 1$, and the existing lower bound is sharp at $n - Z(n)$ moves, where $Z(n)$ is the number of terms in the Zeckendorf decomposition of $n$ and $IZ(n)$ is the sum of indices in the same Zeckendorf decomposition of $n$. We also studied four deterministic variants of the game, where there was a fixed order on which available move one takes: Combine Largest, Split Largest, Combine Smallest and Split Smallest. We prove that Combine Largest and Split Largest realize the lower bound. Split Smallest has the largest number of moves over all possible games, and is close to the new upper bound. For Combine Split games, the number of moves grows linearly with $n$. \end{abstract}


\thanks{The authors were partially supported by NSF grants DMS1265673 and DMS1561945. We thank the Elite Scholars Program for facilitating this collaboration.}

%%%%%%%%%%%%%%%%%%%%%%%%%%%%%%%%%%%%%%%%%%%%%%%%%%%%%%%%%%%%%%%%%%%%%%%%%%%%%%%%%%
%%%%%%%%%%%%%%%%%%%%%%%%%%%%%%%%%%%%%%%%%%%%%%%%%%%%%%%%%%%%%%%%%%%%%%%%%%%%%%%%%%

%\tableofcontents

%%%%%%%%%%%%%%%%%%%%%%%%%%%%%%%%%%%%%%%%%%%%%%%%%%%%%%%%%%%%%%%%%%%%%%%%%%%%%%%%%%%%%%%%%%%%%%%%%%%%%%%%%%%%%%%%%%%%%%%%%%%%%%%%%%%%%%%%%%%%%
%%%%%%%%%%%%%%%%%%INTRODUCTION%%%%%%%%%%%%%%%%%%%%%%%%%%%%%%%%%%%%%%%%%%%%%%%%%%%%%%%%%%%%%%%%%%%%%%%%%%%%%%%%%%%%%%%%%%%%%%%%%%%%%%%%%%%%%%%%%
%%%%%%%%%%%%%%%%%%%%%%%%%%%%%%%%%%%%%%%%%%%%%%%%%%%%%%%%%%%%%%%%%%%%%%%%%%%%%%%%%%%%%%%%%%%%%%%%%%%%%%%%%%%%%%%%%%%%%%%%%%%%%%%%%%%%%%%%%%%%%

\section{Introduction}\label{section}

%%%%%%%%%%%%%%%%%%%%%%%%%%%%%%%%%%%%%%%%%%%%%%%
%%%%%%%%%%%%%%%%%%%%%%%%%%%%%%%%%%%%%%%%%%%%%%%
%%%%%%%%%%%%%%%%%%%%%%%%%%%%%%%%%%%%%%%%%%%%%%%
\subsection{The Zeckendorf Game}
The Fibonacci numbers are among the most interesting and famous sequences; see \cite{Kos} for a collection of some of their properties. We define them by $F_1 = 1, F_2 = 2$ and $F_{n+1} = F_n + F_{n-1}$; with these initial conditions we have Zeckendorf's Theorem \cite{Ze}: every positive integer has a unique representation as a sum of non-adjacent Fibonacci numbers.\footnote{We clearly lose uniqueness if we include $F_0 = 0$ or start $F_1 = F_2 = 1$.} For example, $$2020 \ = \  1597 + 377 + 34 + 8 + 3 + 1 \  = \ F_{16} + F_{13} + F_8 +  F_5 + F_3 + F_1. $$ There is now an extensive literature on proofs of this theorem and generalizations; see for example \cite{Al, Br, CFHMN1, CFHMN2, CHHMPV, Day, DDKMMV, Fr, GTNP, Ha, Ho, Ke, Len, MW1, MW2}.

Baird-Smith, Epstein, Flint and Miller \cite{BEFMfibq, BEFMcant} create a game based on the Zeckendorf decompositions. We quote from \cite{BEFMcant}, describing the game and previous results. We first introduce some notation. By $\{1^n\}$ or $\{{F_1}^n\}$ we mean $n$ copies of $1$, the first Fibonacci number. If we have 3 copies of $F_1$, 2 copies of $F_2$, and 7 copies of $F_4$, we could write either $\{{F_1}^3 \wedge {F_2}^2 \wedge {F_4}^7 \}$ or $\{1^3 \wedge 2^2 \wedge 5^7\}$.

\begin{defi}[The Two Player Zeckendorf Game]\label{defi:zg}
At the beginning of the game, there is an unordered list of $n$ 1's. Let $F_1 = 1, F_2 = 2$, and $F_{i+1} = F_i + F_{i-1}$; therefore the initial list is $\{{F_1}^n\}$. On each turn, a player can do one of the following moves.
\begin{enumerate}
\item If the list contains two consecutive Fibonacci numbers, $F_{i-1}, F_i$, then a player can change these to $F_{i+1}$. We denote this move $\{F_{i-1} \wedge F_i \rightarrow F_{i+1}\}$.
\item If the list has two of the same Fibonacci number, $F_i, F_i$, then
\begin{enumerate}
\item if $i=1$, a player can change $F_1, F_1$ to $F_2$, denoted by $\{F_1 \wedge F_1 \rightarrow F_2\}$,
\item if $i=2$, a player can change $F_2, F_2$ to $F_1, F_3$, denoted by $\{F_2 \wedge F_2 \rightarrow F_1 \wedge F_3\}$, and
\item if $i \geq 3$, a player can change $F_i, F_i$ to $F_{i-2}, F_{i+1}$, denoted by $\{F_i \wedge F_i \rightarrow F_{i-2}\wedge F_{i+1} \}$.
\end{enumerate}
\end{enumerate}
The players alternative moving. The game ends when one player moves to create the Zeckendorf decomposition.
\end{defi}

The moves of the game are derived from the recurrence, either combining terms to make the next in the sequence or splitting terms with multiple copies. The game is well-defined and ends after at most $i_n \cdot n$ moves, where $i$ is the largest index such that $F_i \le n$. Thus the game takes at most order $n \log n$ moves as the Fibonacci numbers grow exponentially fast.\footnote{With our normalization of $F_1 = 1, F_2 = 2$ we have $F_n$ is the closet integer to $\phi^{n+1}/\sqrt{5}$, where $\phi = (1+\sqrt{5})/2$ is the golden mean.} The shortest game takes $n - Z(n)$ moves, where $Z(n)$ is the number of terms in $n$'s Zeckendorf decomposition; this is realized by using a Greedy Algorithm (at each turn one must move on the largest possible index). If $n>2$ then Player Two has a winning strategy, although the proof of this is an existence proof and does not construct a winning strategy.

%%%%%%%%%%%%%%%%%%%%%%%%%%%%%%%%%%%%%%%%%%%%%%%
%%%%%%%%%%%%%%%%%%%%%%%%%%%%%%%%%%%%%%%%%%%%%%%
%%%%%%%%%%%%%%%%%%%%%%%%%%%%%%%%%%%%%%%%%%%%%%%
\subsection{Deterministic Zeckendorf Games}

As the optimal strategy of the Zeckendorf Game remains elusive, we study instead four deterministic games. These are defined by specifying the order in which moves must be done. While there is thus no strategy,\footnote{Our games are equivalent to the classic card game of War, at least under the assumption that you have no freedom in how you pick up the cards.} it is illuminating to study these special cases as it provides some results that clarify some of the behavior of the general game.

The four games are as follows; for each game we list the order in which the moves must be done. By adding 1's we mean the move $F_1 \wedge F_1 \to F_2$.

\begin{itemize}

\item	\emph{Combine largest:} adding consecutive indices from largest to smallest, adding 1's, splitting from largest to smallest.

\item	\emph{Split largest:} splitting from largest to smallest, adding consecutive indices from largest to smallest, adding 1's.
	
\item	\emph{Combine smallest:} adding 1's, adding consecutive indices from smallest to largest, splitting from smallest to largest.
	
\item	\emph{Split smallest:} splitting from smallest to largest, adding 1's, adding consecutive indices from smallest to largest.
	
\end{itemize}

The number of moves of a game is the sum of the number of combining moves and the number of splitting moves. We let $MC_i$ denote the number of combining moves at the index $i$ with $2 \le i$, with of course $MC_1$ the number of adding 1's. Similarly the number of splitting moves at $i$ is denoted $MS_i$ for $i \ge 2$ (note we are considering adding 1's as a combining move and not a splitting one; we explain this choice in Remark \ref{rek:whycombine}). By an abuse of notation, we also refer to combining moves at $i$ by $MC_i$ and splitting moves at $i$ by $MS_i$.

Our main result is a proof of the conjecture from \cite{BEFMfibq, BEFMcant} that the number of moves in any game is linear in $n$.

\begin{thm}\label{thm:lineargrowth} The number of moves in the longest game is bounded by $3n - 3Z(n) - IZ(n) + 1$, where $Z(n)$ is the number of terms in the Zeckendorf decomposition of $n$ and  $IZ(n)$ is the sum of indices in $n$'s Zeckendorf decomposition. As the number of moves is at least $n - Z(n)$, each game takes order $n$ moves to play. \end{thm}

The previous upper bound (order $n\log n$) was already very close to the known lower bound (order $n$), indicating that a new perspective would be needed to close the gap and have them at the same order of magnitude. We quickly sketch the key ideas, and highlight why we are able to remove the logarithmic factor. The starting point is the monovariant introduced in \cite{BEFMfibq, BEFMcant}, which we review and expand on in \S\ref{sec:monovariants}. We use related monovariants to derive a bound growing linearly with $n$ for a weighted sum of all moves in a game except for splitting $F_2$'s or $F_3$'s. We then prove that the number of combining moves is independent of how the game is played, and the number of splitting moves at indices 2 or 3 is related to the number of combining moves at 1 and 2, which we just proved grows linearly with $n$ and completes the proof.

More is true for our deterministic games. We can rigorously determine the behavior of two of the games, and conjecture for the other two, with data strongly supporting those claims.

\begin{thm}\label{thm:movesombinelargestsplitlargest} The Combine Largest and Split Largest games also realize the lower bound.  Both have $MS(n) = 0$ and $MC(n) = n - Z(n)$. \end{thm}

\begin{conj}\label{thm:movessplitsmallcombinesplit} For Split Smallest, the number of moves grows linearly with $n$; numerically the constant appears to be the golden mean squared. For Combine Smallest games, the number of moves grows linearly with $n$, with the constant appearing to be approximately $1.206$.  \end{conj}

We end with two conjectures. In \cite{BEFMfibq, BEFMcant}, it was conjectured that as $n$ goes to infinity, the number of moves in a random game when all legal moves are equally likely converges to a Gaussian. The data suggests the average number of moves is approximately $1.2n$.

\begin{conj}\label{conj:gaussian} The number of splitting moves in a random game converges to a Gaussian, with mean and variance approximately $0.215n$. \end{conj}

\begin{conj}\label{conj:splitsmallestshort}
It was conjectured that the longest game on any $n$ is achieved by applying splitting moves whenever possible. Specifically, the longest possible game applies moves in the following order: adding 1's, splitting from smallest to largest, and adding consecutive indices from smallest to largest. We find another candidate for a longest possible game,  with moves in the following order: splitting from smallest to largest, adding 1's, and adding consecutive indices, from smallest to largest. This is the deterministic game Split Smallest.
\end{conj}

%%%%%%%%%%%%%%%%%%%%%%%%%%%%%%%%%%%%%%%%%%%%%%%%%%%%%%%%%%%%%%%%%%%%%%%%%%%%%%%%%%%%%%%%%%%%%%%%%%%%%%%%%%%%%%%%%%%%%%%%
%%%%%%%%%%%%%%%%%%%%%%%%%%%%%%%%%%%%%%%%%%%%%%%%%%%%%%%%%%%%%%%%%%%%%%%%%%%%%%%%%%%%%%%%%%%%%%%%%%%%%%%%%%%%%%%%%%%%%%%%
%%%%%%%%%%%%%%%%%%%%%%%%%%%%%%%%%%%%%%%%%%%%%%%%%%%%%%%%%%%%%%%%%%%%%%%%%%%%%%%%%%%%%%%%%%%%%%%%%%%%%%%%%%%%%%%%%%%%%%%%
\section{Monovariants}\label{sec:monovariants}

We use two monovariants in our investigation. The first is obvious, and has not been explicitly isolated in earlier work.

\ \\

\noindent \textbf{Monovariant I: the number of terms in the decomposition of $n$ never increases throughout the game.}

\ \\

The proof is immediate, as we only have two types of moves. One combines two terms into one, which decreases by 1 the number of terms in the decomposition of $n$; the other splits a repeated term into two distinct terms, which does not change the number of summands.

\ \\

\noindent \textbf{Monovariant II: The sum of the indices indices in the decomposition of $n$ never increases throughout the game.}

\ \\

This is slightly different than the monovariant used in \cite{BEFMfibq, BEFMcant}, where the sum of the square-root of the indices is studied. The advantage of this quantity is that this sum is strictly decreases, and must decrease by at least a fixed positive amount which is bounded below by a positive function of $n$. Thus one can immediately deduce that the Zeckendorf game terminates. Related monovariants are used to analyze a related property of Zeckendorf decompositions, both for the Fibonacci and other recurrences: of all decompositions of an integer $n$ as a sum of Fibonacci numbers, no decomposition has fewer summands than the Zeckendorf decomposition. In \cite{CHHMPV}, the authors find conditions on recurrence relations where no decomposition has fewer summands than the Generalized Zeckedorf decomposition.

For our purposes, however, it is easier to work with Monovariant II. It is straightforward to show that the sum of the indices never increases. We only have four moves to consider (as moves involving the index 1 are defined slightly differently, we need to study that case separately).

\begin{itemize}

\item Adding consecutive terms: If we replace $F_k \wedge F_{k-1}$ with $F_{k+1}$, for $k \ge 2$, we replace $k + k-1$ with $k+1$, and the sum of the indices has decreased by $k-2 \ge 0$ (it is positive so long as $k \ge 3$).

\item	Adding 1's: If we replace $F_1 \wedge F_1$ with $F_2$ there is no change in the sum of the indices, as $1+1 = 2$.

\item Splitting terms: If we replace $F_k \wedge F_k$ with $F_{k+1} \wedge F_{k-2}$, for $k \ge 3$, we replace $2k$ with $(k+1)+(k-2)$, and the sum of the indices has decreased by 1.

\item Splitting 2's: If we replace $2F_2$ with $F_3\wedge F_1$ then there is no change in the sum of the indices, as $2\cdot 2 = 3 + 1$.

\end{itemize}

%%%%%%%%%%%%%%%%%%%%%%%%%%%%%%%%%%%%%%%%%%%%%%%%%%%%%%%%%%%%%%%%%%%%%%%%%%%%%%%%%%%%%%%%%%%%%%%%%%%%%%%%%%%%%%%%%%%%%%%%
%%%%%%%%%%%%%%%%%%%%%%%%%%%%%%%%%%%%%%%%%%%%%%%%%%%%%%%%%%%%%%%%%%%%%%%%%%%%%%%%%%%%%%%%%%%%%%%%%%%%%%%%%%%%%%%%%%%%%%%%
%%%%%%%%%%%%%%%%%%%%%%%%%%%%%%%%%%%%%%%%%%%%%%%%%%%%%%%%%%%%%%%%%%%%%%%%%%%%%%%%%%%%%%%%%%%%%%%%%%%%%%%%%%%%%%%%%%%%%%%%
\section{Length of Games}\label{sec:lengthofgames}

We use the two monovariants introduced in \S\ref{sec:monovariants} to derive bounds for weighted sums of the number of each type of move; our theorems are then immediate consequences. We first isolate some lemmas which will be useful in the proof.

%%%%%%%%%%%%%%%%%%%%%%%%%%%%%%%%%%%%%%%%%%%%%%%
%%%%%%%%%%%%%%%%%%%%%%%%%%%%%%%%%%%%%%%%%%%%%%%
%%%%%%%%%%%%%%%%%%%%%%%%%%%%%%%%%%%%%%%%%%%%%%%
\subsection{Preliminary Lemmas}

When the game starts, the sum of the indices is $n$. From \cite{BEFMfibq, BEFMcant} the game always ends in $n$'s Zecknedorf decomposition, which by definition has $IZ(n)$ summands. Let $i_{\max}(n)$ be the largest index $m$ such that $F_m \le n$; this will be the largest index in $n$'s decomposition, and by Binet's formula we know \begin{equation} F_m \ =\ \frac{1}{\sqrt{5}}\left(\frac{1+\sqrt{5}}{2}\right)^{m+1} -  \frac{1}{\sqrt{5}}\left(\frac{1-\sqrt{5}}{2}\right)^{m+1} \ = \ \frac{\phi^{m+1}}{\sqrt{5}} - \frac{(1-\phi)^{m+1}}{\sqrt{5}}, \end{equation}  where $\phi$ is the Golden mean. As the second term above exponentially decays to zero, we have $F_m \approx \phi^{m+1}/\sqrt{5}$. Thus the largest index is bounded by $\lfloor\log_{\phi}(n \sqrt{5})-1\rfloor + 1$, or \begin{equation}\label{eq:imax} i_{\max}(n)\ \le\  \log_{\phi}(n\sqrt{5}), \end{equation} which is of order $\log_{\phi}(n)$.

%%% F[m_] := Simplify[(((1 + Sqrt[5])/2)^(m + 1) - ((1 - Sqrt[5])/2)^(m + 1) ) / Sqrt[5]]

\begin{lemma} The sum of the indices in $n$'s decomposition, $IZ(n)$, is bounded by a constant multiple of $\log_\phi^2(n)$: \begin{equation} IZ(n) \ \le \   \frac{(\log_{\phi}(n\sqrt{5}) + 3)^2}{2}. \end{equation} \end{lemma}

\begin{proof}
As the Zeckendorf decomposition cannot have adjacent summands, the maximum sum of indices is the sum of every other term to $i_{\max}(n)$. To be safe, we start with index 2 and end at $i_{\max}(n) + 2$; this negligibly increases the sum as we are only adding one more term, while the sum is on the order of the square of the largest summand. We have
\begin{eqnarray} 2 + 4 + \cdots + (i_{\max}(n) + 2) & \ = \ & 2\left(1 + 2 + \cdots + \frac{i_{\max}(n) + 2}{2}\right) \nonumber\\ & \ = \ & 2\frac{i_{\max}(n) + 2}{2} \frac{i_{\max}(n) + 4}{2} \ < \ \frac{(i_{\max}(n) + 3)^2}{2}. \end{eqnarray} From \eqref{eq:imax} we know $i_{\max}(n) \le \log_{\phi}(n\sqrt{5})$; thus \begin{equation} IZ(n) \ \le \ \frac{(\log_{\phi}(n\sqrt{5}) + 3)^2}{2}, \end{equation} so $IZ(n)$ is bounded by a constant multiple of $\log_{\phi}^2(n)$.
\end{proof}

We now prove the first of the two useful results on weighted sums of indices.

\begin{lemma}\label{lem:mc3tomsmax} We have
\begin{equation} MC_3  + 2MC_4 + \cdots + (i_{\max}(n) -2)) MC_{i_{\max}(n)}  + MS_3 + MS_4  + \cdots +  MS_{i_{\max}(n)}  \ = \ n - IZ(n). \end{equation} In particular, for $2 \le k \le i_{\max}(n)$ we have \begin{equation} MC_k \ \le\ \frac{n-IZ(n)}{k-2}.\end{equation}
\end{lemma}

\begin{proof}
When the game starts we have $n$ copies of $F_1$, for an index sum of $n$; the game ends in $n$'s Zeckendorf decomposition, with index sum of $IZ(n)$. Thus the change in the index sum, $n - IZ(n)$, must equal the change from each move. We now compute that change by looking at what happens at each index.

From \S\ref{sec:monovariants}, the index sum changes by $k-2$ when we combine at index $k$ if $k \ge 2$. Thus the contribution from all these moves is $(k-2)M_k$; note there is no contribution when $k=2$. There is no change in the index sum from adding 1's, so $MC_1$ will not appear in the relation; similarly $MS_2$ will not appear as that moves also does not change the index sum. If $k \ge 3$ then the splitting move at $k$ decreases the index sum by 1, so the contribution of all the splitting moves at $k$ is simply $MS_k$. Combining, we find \begin{equation} MC_3 + 2 MC_4 + 3 MC_5 + \cdots + (i_{\max}(n) -2)) MC_{i_{\max}(n)}  + MS_3 + \cdots +  MS_{i_{\max}(n)}  \ = \ n - IZ(n). \end{equation} The bound on $MC_k$ follows immediately; as we will show later that the number of moves is at most $3n$, we can obtain similar linear in $n$ bounds for $MC_1$ and $MC_2$.
\end{proof}

\begin{lemma}\label{lem:mcn} The number of combining moves in a game, $MC(n)$, is independent of how the game is played, and \begin{equation}\label{eq:mcnsum}
MC(n) \ :=\  MC_1  +  MC_2 + \cdots + MC_{i_{\max}(n)} \  = \  n- Z(n), \end{equation} where $Z(n)$ is the number of terms in $n$'s Zeckendorf decomposition. \end{lemma}

\begin{proof} As we start with $n$ indices and end with $Z(n)$ indices, the change in indices throughout the game is $n-Z(n)$. Each splitting move leaves the number of terms in $n$'s decomposition alone, while each combining move decreases the number of terms by one; \eqref{eq:mcnsum} follows immediately, and we see in particular this quantity is independent of how the game is played.
\end{proof}

\begin{rek}\label{rek:whycombine} The lemma above is why we view $F_1 \wedge F_1 = F_2$ as a combining move and not a splitting move $F_1 \wedge F_1 = F_2 \wedge F_{-1}$. Not only do we not have $F_{-1}$ as an available summand, but with this definition all combining moves decrease the number of summands by 1 while each splitting move leaves the number unchanged. \end{rek}

%%%%%%%%%%%%%%%%%%%%%%%%%%%%%%%%%%%%%%%%%%%%%%%
%%%%%%%%%%%%%%%%%%%%%%%%%%%%%%%%%%%%%%%%%%%%%%%
%%%%%%%%%%%%%%%%%%%%%%%%%%%%%%%%%%%%%%%%%%%%%%%
\subsection{Proof of Theorem \ref{thm:lineargrowth}}\label{sec:lineargrowth}

We now prove our main result.

\begin{proof}[Proof of Theorem \ref{thm:lineargrowth}]
Lemmas \ref{lem:mc3tomsmax} and \ref{lem:mcn} almost suffice for the proof, as they provide bounds for weighted sums of $MC_1, \dots, MC_{i_{\max}(n)}, MS_4, \dots, MS_{i_{\max}(n)}$. From Lemma \ref{lem:mc3tomsmax} we have
\begin{equation} MC_3  + 2MC_4 + \cdots + (i_{\max}(n) -2)) MC_{i_{\max}(n)}  + MS_3 + MS_4  + \cdots +  MS_{i_{\max}(n)}  \ = \ n - IZ(n), \end{equation} and thus
\begin{equation}\label{eq:main1} MC_3 + MC_4 + \cdots + MC_{i_{\max}(n)}  + MS_3 + MS_4  + \cdots +  MS_{i_{\max}(n)} \ \le \ n - IZ(n). \end{equation} Note the sum above misses $MC_1, MC_2$ and $MS_2$.

We can easily bound $MC_1 + MC_2$ from Lemma \ref{lem:mcn}: \begin{equation}
MC(n) \ :=\  MC_1  +  MC_2 + \cdots + MC_{i_{\max}(n)} \  = \  n- Z(n); \end{equation} thus $MC_1 + MC_2 \le n - Z(n)$, and adding this to \eqref{eq:main1} yields \begin{equation}\label{eq:main2} MC_1 + \cdots + MC_{i_{\max}(n)}  + MS_3 + \cdots +  MS_{i_{\max}(n)} \ \le \ 2n - Z(n) - IZ(n). \end{equation}

All that remains is to bound $MS_2$. We start the game with $n$ copies of $F_1$, and let $\delta_1$ be how many $F_1$'s we have when the game terminates in $n$'s Zeckendorf decomposition; clearly $\delta_1 \in \{0, 1\}$ as we cannot have a repeated summand from the definition of the Zeckendorf decomposition, and its value is independent of how the game is played. Every time we combine two 1's to make a 2 ($F_1 \wedge F_1 = F_2$) we decrease the number of $F_1$'s by 2, while every time we combine a 1 and a 2 to make a 3 ($F_1 \wedge F_2 = F_3$) we decrease the number of $F_1$'s by 1. These are the only moves that decrease the number of $F_1$'s; the only moves that increase the number of $F_1$'s are the splitting moves at 1 and 2. Each splitting move at 2 ($2 F_2 = F_3 \wedge F_1$) increases the number of $F_1$'s by 1, as does each splitting move at 3 ($2F_3 = F_4 \wedge F_1$). Thus \begin{equation} n - 2MC_1 - 2MC_2 + MS_2 + MS_3 \ = \ \delta_1, \ \ \ {\rm or} \ \ \ MS_2 + MS_3 \ = \ 2MC_1 + 2MC_2 - n + \delta. \end{equation} By Lemma \ref{lem:mcn} we can bound $MC_1 + MC_2$ by $MC(n) \le n - Z(n)$, and thus \begin{equation} MS_2 \ \le \ 2(n - Z(n)) - n + 1 \ = \ n - 2 Z(n) + 1. \end{equation}

Combining our bound for $MS_2$ with what we have established in \eqref{eq:main2} yields
\begin{equation} MC_1 + \cdots + MC_{i_{\max}(n)} + MS_1 + \cdots + MS_{i_{\max}(n)} \ \le \ 3n - 3Z(n) - IZ(n) + 1, \end{equation} which is essentially of size $3n$ as $Z(n)$ is of order $\log_\phi(n)$ and $IZ(n)$ is at most $\log_\phi^2(n)$.

This proves the upper bound for any game grows linearly with $n$; as every game exactly $MC(n) = n - Z(n)$ combining moves, the lower bound is also linear in $n$, completing the proof.
\end{proof}

%%%%%%%%%%%%%%%%%%%%%%%%%%%%%%%%%%%%%%%%%%%%%%%
%%%%%%%%%%%%%%%%%%%%%%%%%%%%%%%%%%%%%%%%%%%%%%%
%%%%%%%%%%%%%%%%%%%%%%%%%%%%%%%%%%%%%%%%%%%%%%%
\subsection{Analysis of Combine Largest and Split Largest}

\begin{proof}[Proof of Theorem \ref{thm:movesombinelargestsplitlargest}] We prove that the Combine Largest and Split Largest games also realize the lower bound. Remember that for each game the order in which we do moves is as follows:

\begin{itemize}

\item	\emph{Combine largest:} adding consecutive indices from largest to smallest, adding 1's, splitting from largest to smallest.

\item	\emph{Split largest:} splitting from largest to smallest, adding consecutive indices from largest to smallest, adding 1's.

\end{itemize}

The claims follow if we can show in each game we never have a splitting move, as we showed in Lemma \ref{lem:mcn} that the number of combining moves is always $MC(n) = n - Z(n)$.

We show that for these two games, at any state in the game if  $F_i^j$ is in the current decomposition of $n$ and $i \ge 2$, then $j = 1$. We proceed by induction. The initial list, our base case, is $\{F_1^n\}$, so the claim is true before the first move.  The first move must be adding 1's: $F_1^n  \to F_1^{n-2} \wedge  F_2$. After this move, the statement is still true. Though not necessary, we can check and see that the claim is still true after the second and third moves.

We now turn to the inductive step; we may assume that if you look at the decomposition of $n$ then for any index $i$ if $F_i^j$ is in the list, then $j = 1$.  Since there is no possible splitting move as each index appears at most once, our move must be chosen in the following order: adding consecutive terms from largest to smallest, and adding 1's.

\emph{Case 1: Consecutive Terms:} If the list contains two or more pairs of consecutive Fibonacci numbers, the largest $F_i$ is selected for the combining move. Note we \emph{cannot} have $F_{i+1}$ in our list, as if that were the case we would have chosen $F_{i+1}$ as the largest term. Thus after this move we still have each index $i \ge 2$ appearing at most once.

\emph{Case 2: No Consecutive Terms:} If $n$'s decomposition does not contain consecutive Fibonacci numbers, then the only possible move is adding 1's. Note that we \emph{cannot} have $F_2$ in our list, as if we did we would have combined $F_1$ and $F_2$. Thus after combining two 1's we still do not have any index $i \ge 2$ occurring more than once.
\end{proof}

%%%%%%%%%%%%%%%%%%%%%%%%%%%%%%%%%%%%%%%%%%%%%%%%%%%%%%%%%%%%%%%%%%%%%%%%%%%%%%%%%%%%%%%%%%%%%%%%%%%%%%%%%%%%%%%%%%%%%%%%
%%%%%%%%%%%%%%%%%%%%%%%%%%%%%%%%%%%%%%%%%%%%%%%%%%%%%%%%%%%%%%%%%%%%%%%%%%%%%%%%%%%%%%%%%%%%%%%%%%%%%%%%%%%%%%%%%%%%%%%%
%%%%%%%%%%%%%%%%%%%%%%%%%%%%%%%%%%%%%%%%%%%%%%%%%%%%%%%%%%%%%%%%%%%%%%%%%%%%%%%%%%%%%%%%%%%%%%%%%%%%%%%%%%%%%%%%%%%%%%%%
\section{Conjectures on the Number of Moves}

Conjecture \ref{thm:movessplitsmallcombinesplit} states that for Split Smallest, the number of moves grows linearly with $n$  (numerically the constant appears to be the golden mean squared), while for Combine Smallest games, the number of moves grows linearly with $n$ (with the constant appearing to be approximately $1.206$). We arrived at these results from analyzing large numbers of games; we provide representative examples in Figures \ref{fig:splitsmallest1600000plot} and \ref{fig:combinesmallest1600000plot}.

%Note for the Split Smallest we have a conjectured form of the constant, while for Combine Largest we are unsure precisely how it is related to the Golden Mean.

\begin{figure}[h]
\begin{center}
\scalebox{.475}{\includegraphics{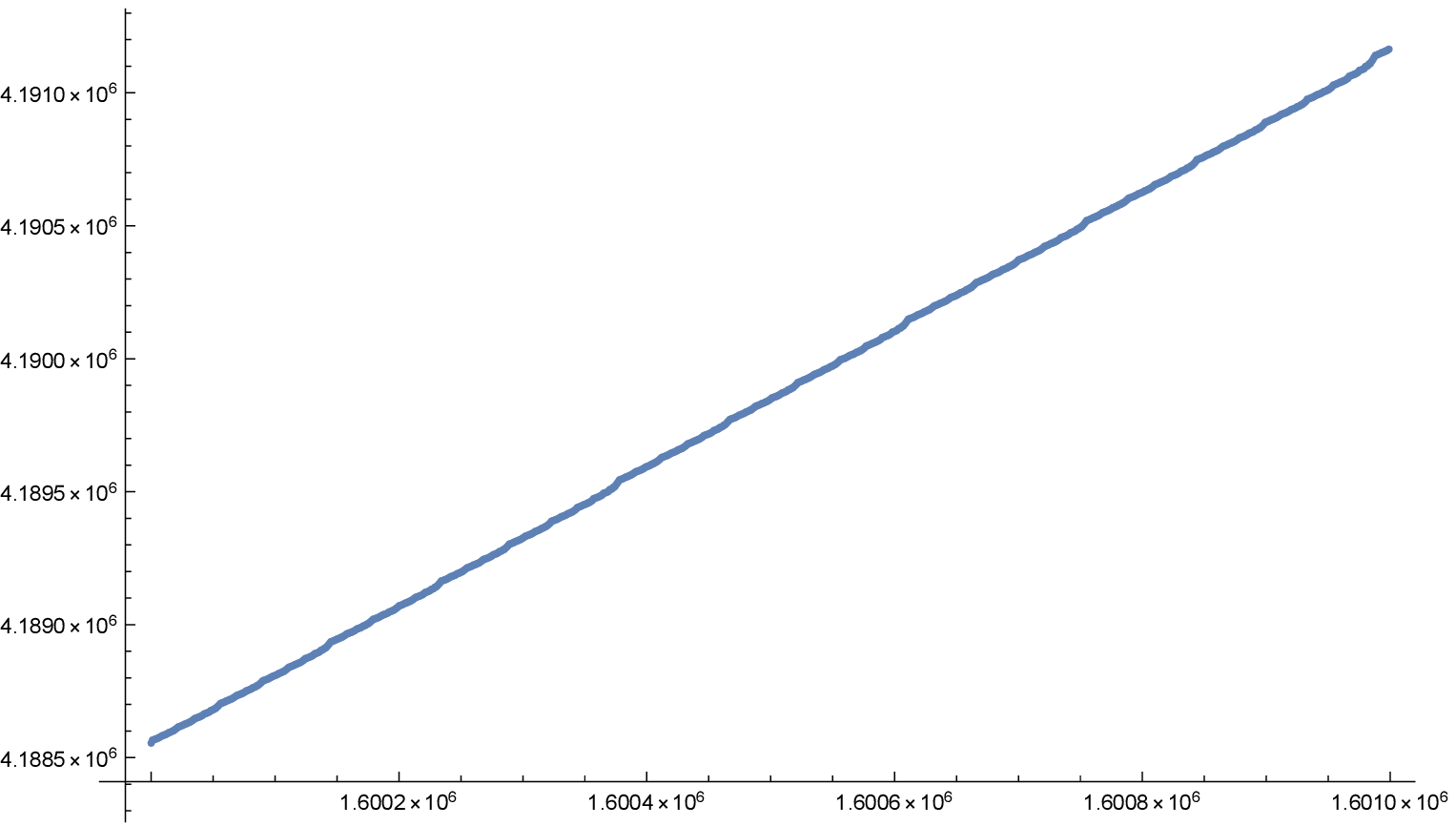}} \ \ \scalebox{.475}{\includegraphics{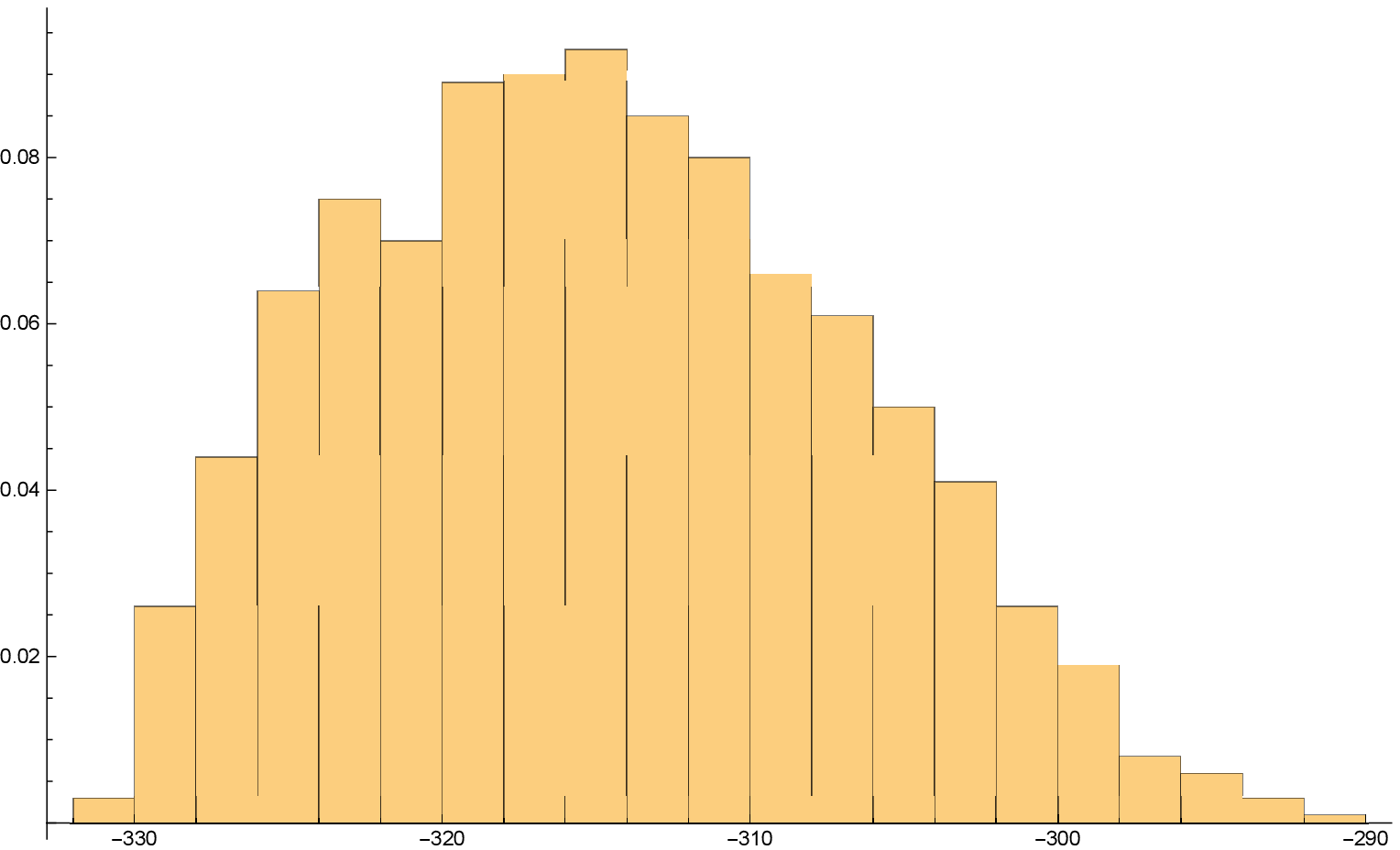}}
\caption{\label{fig:splitsmallest1600000plot} Results of deterministic game Split Smallest for 1000 consecutive $n$, starting at 1,600,000. Left: Plot of the number of moves versus $n$. Right: Histogram of number of moves for $n$  minus $\phi^2 n$, where $\phi = (1+\sqrt{5})/2$ is the Golden Mean.}
\end{center}\end{figure}

\begin{figure}[h]
\begin{center}
\scalebox{.475}{\includegraphics{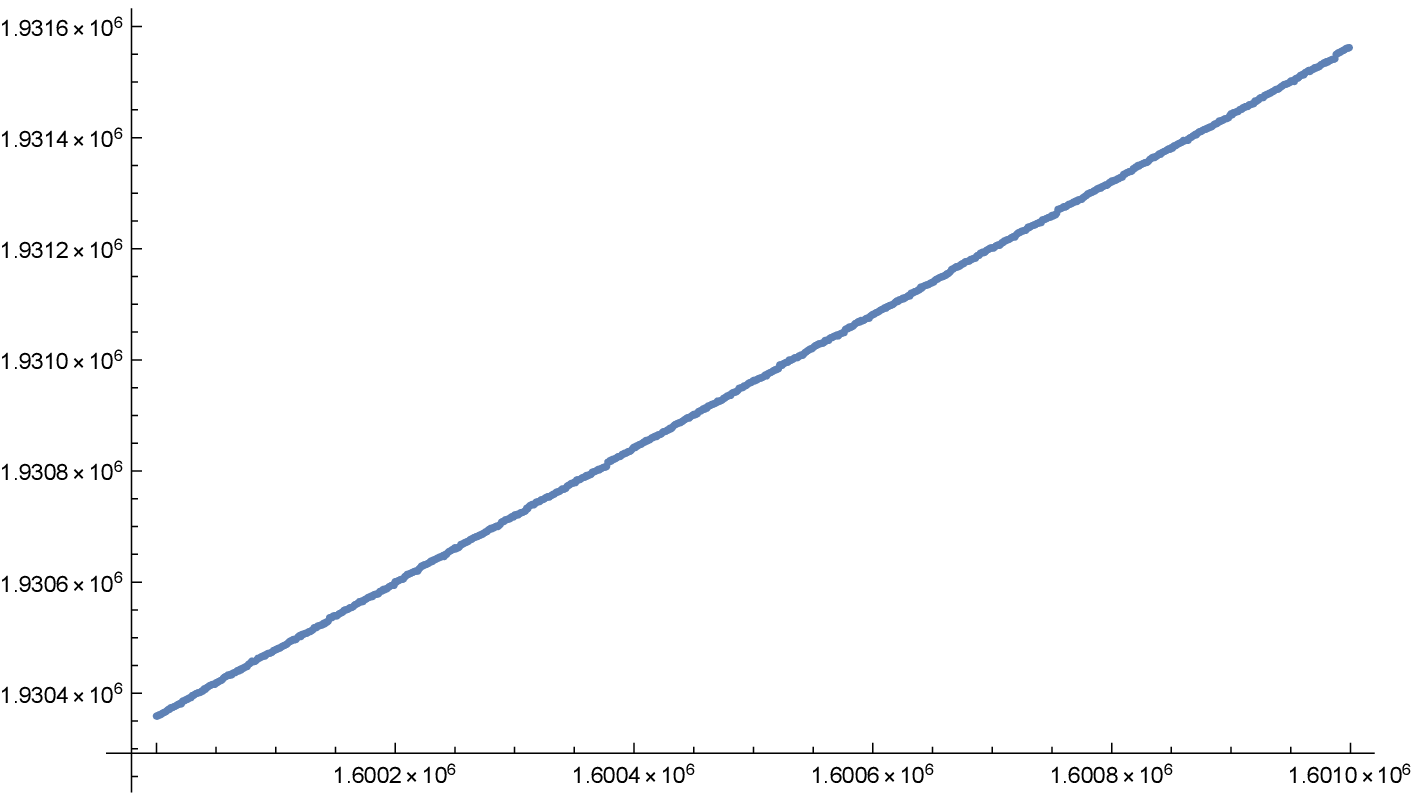}} \ \ \scalebox{.475}{\includegraphics{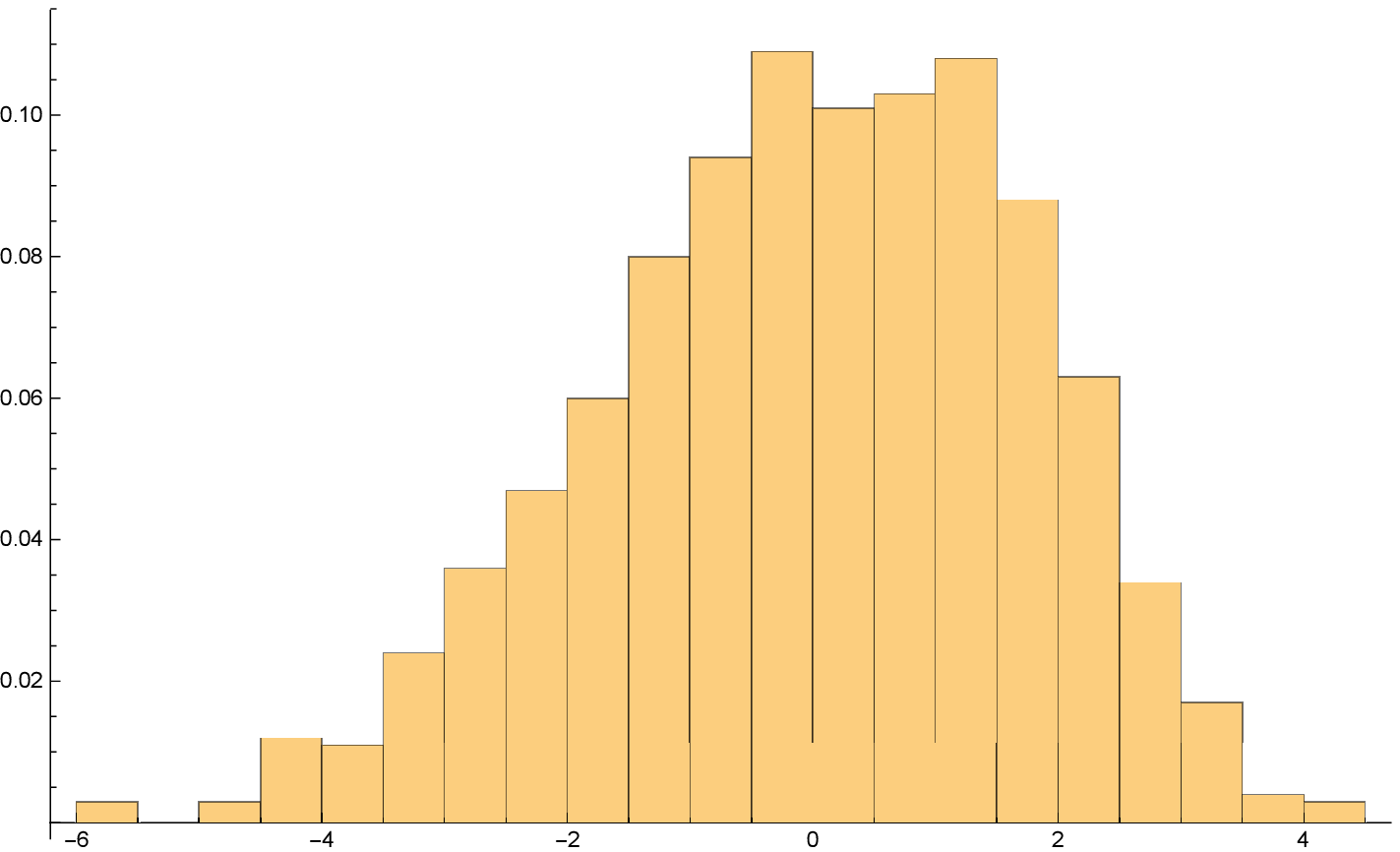}}
\caption{\label{fig:combinesmallest1600000plot} Results of deterministic game Combine Smallest for 1000 consecutive $n$, starting at 1,600,000. Left: Plot of the number of moves versus $n$. Right: Histogram of number of moves for $n$  minus $1.20647 n$.}
\end{center}\end{figure}

We then looked at random games. By Lemma \ref{lem:mcn} the number of combining moves in a game depends only on $n$; thus all that varies is the number of splitting moves. We played 10,000 games with $n$ equal to one million, and display the results in Figure \ref{fig:gaussian1000000} plotted against a Gaussian. The fit is very good, providing support for Conjecture \ref{conj:gaussian}.

\begin{figure}[h]
\begin{center}
\scalebox{1}{\includegraphics{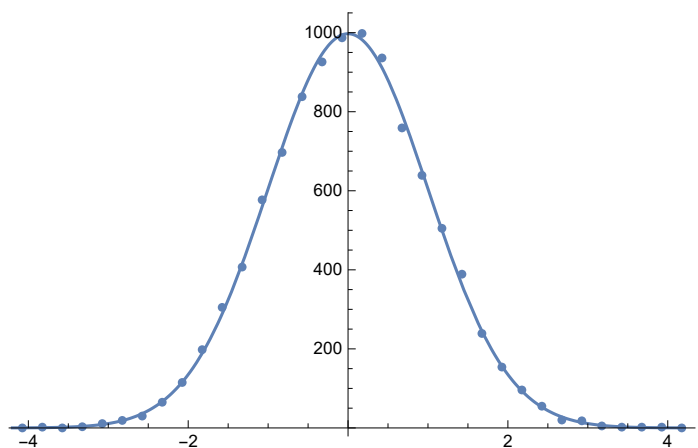}}
\caption{\label{fig:gaussian1000000} Plot of 10,000 random games with $n = 1,000,000$; the data was standardized to have mean 0 and variance 1, and we overlay a plot of the standard normal.}
\end{center}\end{figure}

Finally, our investigations of the four deterministic games, and Lemma \ref{lem:mcn} which states that all games have the same number of combining moves, provides support for Conjecture \ref{conj:splitsmallestshort}. The move counts from the two deterministic algorithms in the conjecture were identical for all $n < 1,600,000$.

We wrote a Java program to explore Zeckendorf games, available at \begin{center} \small \bburl{https://web.williams.edu/Mathematics/sjmiller/public_html/math/papers/ZGame.zip}.\normalsize \end{center} In enumerating all games with $n \le 150$ we found the two deterministic games of Conjecture \ref{conj:splitsmallestshort} always had the largest number of moves among all games for a given $n$.

%%%%%%%%%%%%%%%%%%%%%%%%%%%%%%%%%%%%%%%%%%%%%%%%%%%%%%%%%%%%%%%%%%%%%%%%%%%%%%%%%%%%%%%%%%%%%%%%%%%%%%%%%%%%%%%%%%%%%%%%
%%%%%%%%%%%%%%%%%%%%%%%%%%%%%%%%%%%%%%%%%%%%%%%%%%%%%%%%%%%%%%%%%%%%%%%%%%%%%%%%%%%%%%%%%%%%%%%%%%%%%%%%%%%%%%%%%%%%%%%%
%%%%%%%%%%%%%%%%%%%%%%%%%%%%%%%%%%%%%%%%%%%%%%%%%%%%%%%%%%%%%%%%%%%%%%%%%%%%%%%%%%%%%%%%%%%%%%%%%%%%%%%%%%%%%%%%%%%%%%%%

\section{Future Work}

There are many questions related to the Zeckendorf game which can be investigated; several of these will be done by students of Miller in the 2020 PolymathREU. These include the following.

\begin{itemize}

\item The lower and upper bound on game lengths differ by essentially a factor of 3; what is true about the number of moves in most games or in a random game?

\item From \cite{BEFMfibq, BEFMcant} we know that if $n > 2$ then Player Two has a winning strategy; what is it?

\item What if there are $p$ players; what can you say about winning strategies for various $n$ and $p$?

\item The number of combining moves is always $n - Z(n)$, while the number of splitting moves appears to grow linearly with $n$, at approximately $1.206 n$; prove the latter.

\item For each of the four deterministic games, how long do they take and who wins as a function of $n$?

\end{itemize}

%%%%%%%%%%%%%%%%%%%%%%%%%%%%%%%%%%%%%%%%%%%%%%%%%%%%%%%%%%%%%%%%%%%%%%%%%%%%%%%%%%%%%%%%%%%%%%%%%%%%%%%%%%%%%%%%%%%%%%%%
%%%%%%%%%%%%%%%%%%%%%%%%%%%%%%%%%%%%%%%%%%%%%%%%%%%%%%%%%%%%%%%%%%%%%%%%%%%%%%%%%%%%%%%%%%%%%%%%%%%%%%%%%%%%%%%%%%%%%%%%
%%%%%%%%%%%%%%%%%%%%%%%%%%%%%%%%%%%%%%%%%%%%%%%%%%%%%%%%%%%%%%%%%%%%%%%%%%%%%%%%%%%%%%%%%%%%%%%%%%%%%%%%%%%%%%%%%%%%%%%%

\section{Future Work}

Studies of Zeckerdorff game can be extended in many more ways. This paper covered the Zeckendorf Game quite extensively. We proposed the approach to study combining/adding moves and splitting moves separately, the improved upper bounds was found on the number of moves in any game, in the end a few deterministic algorithms were analyzed. We leave the general Player 2 winning strategy for future research.

%%%%%%%%%%%%%%%%%%%%%%%%%%%%%%%%%%%%%%%%%%%%%%%%%%%%%%%%%%%%%%%%%%%%%%%%%%%%%%%%%%%%%%%%%%%%%%%%%%%%%%%%%%%%%%%%%%%%%%%%
%%%%%%%%%%%%%%%%%%%%%%%%%%%%%%%%%%%%%%%%%%%%%%%%%%%%%%%%%%%%%%%%%%%%%%%%%%%%%%%%%%%%%%%%%%%%%%%%%%%%%%%%%%%%%%%%%%%%%%%%
%%%%%%%%%%%%%%%%%%%%%%%%%%%%%%%%%%%%%%%%%%%%%%%%%%%%%%%%%%%%%%%%%%%%%%%%%%%%%%%%%%%%%%%%%%%%%%%%%%%%%%%%%%%%%%%%%%%%%%%%

\medskip

%%%\noindent MSC2010: 11G05 (primary), 11G07, 11G40, 11M41 (secondary)

\end{document}